\documentclass[a4paper,12pt]{article}

\usepackage[utf8]{inputenc}
\usepackage[english]{babel}
\usepackage{amsmath, amssymb, amsthm}
\usepackage{enumerate}
\usepackage[colorlinks=true,linkcolor=blue,citecolor=blue]{hyperref}

\usepackage{tikz}
\usepackage{mdframed}
\usepackage{pgfplots}
\pgfplotsset{width=7cm,compat=1.10}
\pgfplotsset{
  /pgfplots/colormap={pink}{%
    color(0cm) = (blue);
    color(1cm) = (cyan!50!blue);
    color(2cm) = (cyan!50);
    color(3cm) = (cyan) }
}
\usetikzlibrary{calc,arrows,decorations.pathreplacing,fadings,3d,positioning}

\title{On the spectrum of Diophantine exponents of lattices}
\author{Oleg\,N.\,German}
\date{}

\theoremstyle{definition}
\newtheorem{definition}{Definition}

\newtheorem*{notation*}{Notation}

\theoremstyle{remark}

\newtheorem*{remark*}{Remark}

\theoremstyle{plain}
\newtheorem{theorem}{Theorem}
\newtheorem*{theorem*}{Theorem}
\newtheorem{lemma}{Lemma}
\newtheorem*{lemma*}{Lemma}

\newtheorem*{corollary*}{Corollary}

\renewcommand{\phi}{\varphi}

\renewcommand{\vec}[1]{\mathbf{#1}}
\renewcommand{\geq}{\geqslant}
\renewcommand{\leq}{\leqslant}

\newcommand{\e}{\varepsilon}

\newcommand{\R}{\mathbb{R}}
\newcommand{\Z}{\mathbb{Z}}
\newcommand{\Q}{\mathbb{Q}}
\newcommand{\N}{\mathbb{N}}

\newcommand{\La}{\Lambda}
\newcommand{\Ga}{\Gamma}

\newcommand{\cB}{\mathcal{B}}

\newcommand{\cF}{\mathcal{F}}

\newcommand{\cH}{\mathcal{H}}

\newcommand{\cL}{\mathcal{L}}
\newcommand{\cM}{\mathcal{M}}
\newcommand{\cN}{\mathcal{N}}

\newcommand{\cP}{\mathcal{P}}
\newcommand{\cQ}{\mathcal{Q}}
\newcommand{\cR}{\mathcal{R}}
\newcommand{\cS}{\mathcal{S}}
\newcommand{\cT}{\mathcal{T}}
\newcommand{\cV}{\mathcal{V}}

\newcommand{\bB}{\mathbf{B}}
\newcommand{\bE}{\mathbf{E}}
\newcommand{\bL}{\mathbf{L}}

\newcommand{\Gl}{\textup{GL}}

\begin{document}

\maketitle

\begin{abstract}
  In this paper we describe the spectrum of values of weak uniform Diophantine exponents of lattices in arbitrary dimension.
\end{abstract}

\section{Diophantine exponents of lattices}

Let $\La$ be a full rank lattice in $\R^d$. There is a number of problems concerning the question how close to zero the product of a nonzero lattice point coordinates can be. For instance, in the case of an algebraic lattice, i.e. the lattice of a complete module in a totally real algebraic extension of $\Q$ (see \cite{borevich_shafarevich}), such a product is bounded away from zero. But if the product of a nonzero lattice point coordinates can be arbitrarily small, a natural question arises, how fast this product can tend to zero. Simplest quantitative characteristics of this phenomenon are provided by \emph{Diophantine exponents of a lattice}.

For each $\vec x=(x_1,\ldots,x_d)\in\R^d$, let us set
\[
  |\vec x|=\max_{1\leq i\leq d}|x_i|,
  \qquad
  \Pi(\vec x)=\prod_{\begin{subarray}{c}1\leq i\leq d\end{subarray}}|x_i|^{1/d}.
\]

Consider an analogue of an irrationality measure function for the lattice $\La$
\[
  \psi_\La(t)=\min_{\substack{ \vec x\in\La \\ 0<|\vec x|\leq t}}\Pi(\vec x).
\]

\begin{definition} \label{def:lattice_exponents_liminf_limsup}
  Let $\La$ be a full-rank lattice in $\R^d$. The quantity
  \[
    \omega(\La)=\sup\Big\{ \gamma\in\R \,\Big|\, \liminf_{t\to+\infty}\big(t^\gamma\psi_\La(t)\big)<+\infty \Big\}
  \]
  is called the \emph{(regular) Diophantine exponent} of $\La$.
  The quantity
  \[
    \bar\omega(\La)=\sup\Big\{ \gamma\in\R \,\Big|\, \limsup_{t\to+\infty}\big(t^\gamma\psi_\La(t)\big)<+\infty \Big\}
  \]
  is called the \emph{(weak) uniform Diophantine exponent} of $\La$.
\end{definition}

The reason why there is the term ``weak'' in the name of $\bar\omega(\La)$ is that there are at least two ways to define a uniform analogue of the regular exponent, which lead to \emph{strong} and \emph{weak} uniform exponents, respectively. The details can be found in the papers \cite{german_comm_math_2023} and \cite{german_sbornik_2026}, where it is proved that the strong variant of a uniform exponent attains only trivial values in any dimension, whereas the weak uniform exponent, in the two-dimensional case, attains all the values in the segment $[0,+\infty]$. The aim of the current paper is to describe the spectrum of all attainable values of $\bar\omega(\La)$ in an arbitrary dimension.

Minkowski's convex body theorem and the definitions of Diophantine exponents of lattices immediately imply the inequalities
\[
  \omega(\La)\geq
  \bar\omega(\La)\geq0.
\]
Until recently, what was known in dimension $d\geq3$ about the spectrum of $\omega(\La)$, i.e. about the set
\[
  \pmb\Omega_d=
  \Big\{ \omega(\La) \,\Big|\, \text{$\La$ is a full-rank lattice in $\R^d$} \Big\},
\]
was that it contains the ray $[3-d(d-1)^{-2},+\infty]$, the point $0$ and a certain finite set of points (see \cite{german_lattice_exponents_izv_2020}, \cite{german_lattice_transference}, \cite{german_monatshefte_2022}, \cite{german_UMN_2023}).
However, it was natural to expect that for every $d$ the spectrum of $\pmb\Omega_d$ coincides with the ray $[0,+\infty]$. Very recently, this problem was solved by Nikolay Moshchevitin in \cite{moshchevitin_lattice_spectrum}. Thus, it is due to Moshchevitin that we know now that the following holds.

\begin{theorem}\label{t:spectrum_of_the_regular_arbitrary_d}
  For every $d\geq2$, we have $\pmb\Omega_d=[0,+\infty]$.
\end{theorem}

In this paper we prove a similar statement for $\bar\omega(\La)$. Set
\[
  \pmb{\bar\Omega}_d=
  \Big\{ \bar\omega(\La) \,\Big|\, \text{$\La$ is a full-rank lattice in $\R^d$} \Big\}.
\]
The following statement is the main result of the paper.

\begin{theorem}\label{t:spectrum_of_the_weak_arbitrary_d}
  For every $d\geq2$, we have $\pmb{\bar\Omega}_d=[0,+\infty]$.
\end{theorem}

We note that in the proof of Theorem \ref{t:spectrum_of_the_weak_arbitrary_d} we use essentially the ideas contained in Moshchevitin's paper \cite{moshchevitin_lattice_spectrum}.

The paper is organised as follows. In Section \ref{sec:hyperbolic_minima} we define hyperbolic minima -- a very useful tool for multiplicative problems. In Section \ref{sec:2dim_and_additional_forms} we study two-dimensional multiplicative problems that help us to work in higher dimensions. Sections \ref{sec:rank_2_sublattice_construction} and \ref{sec:complement_to_La} are devoted to constructing a lattice in $\R^d$ of rank $2$ and complementing it to a lattice of rank $d$. In Section \ref{sec:metric_lemma} we prove the key statement of metric nature, which enables the appropriate complementing of the lattice of rank $2$ to a lattice of rank $d$. Finally, in Section \ref{sec:proofs} we prove Theorem \ref{t:spectrum_of_the_weak_arbitrary_d} and show how Moshchevitin's Theorem \ref{t:spectrum_of_the_regular_arbitrary_d} can be proved with the help of the constructions developed.

\section{Hyperbolic minima of lattices}\label{sec:hyperbolic_minima}

In many classical problems in the theory of Diophantine approximation, an important role is played by \emph{best approximations}, defined for each setting in their own way. In the case of problems concerning Diophantine exponents of lattices, we are dealing with the essentially multiplicative nature of these problems. For this reason, the role of the best approximations is played by the so-called \emph{hyperbolic minima}.

\begin{definition}\label{def:hyperbolic_minimum}
   Let$\La$ be a full-rank lattice in $\R^d$. We call a nonzero point $\vec x\in\La$ a \emph{hyperbolic minimum} of $\La$ if there are no nonzero points $\vec y\in\La$ such that
  \[
    |\vec y|\leq|\vec x|
    \qquad\text{ and }\qquad
    \Pi(\vec y)<\Pi(\vec x).
  \] 
\end{definition}

Note that the choice of the sup-norm $|\,\cdot\,|$ is fairly conventional. Choosing any other norm does not affect the values of the Diophantine exponents, as all the norms in $\R^d$ are equivalent and a homothetic transformation of a lattice preserves its exponents. And in some problems, it is very convenient to use a norm distinct from the sup-norm. For this reason we give a modification of Definition \ref{def:hyperbolic_minimum}.

\begin{definition}\label{def:hyperbolic_minimum_arbitrary_norm}
  Let$\La$ be a full-rank lattice in $\R^d$ and let $|\,\cdot\,|_\ast$ be an arbitrary norm in $\R^d$. We call a nonzero point $\vec x\in\La$ a \emph{hyperbolic minimum} of $\La$ \emph{w.r.t.} $|\,\cdot\,|_\ast$ if there are no nonzero points $\vec y\in\La$ such that
  \[
    |\vec y|_\ast\leq|\vec x|_\ast
    \qquad\text{ and }\qquad
    \Pi(\vec y)<\Pi(\vec x).
  \] 
\end{definition}

Let us note for the future that, since all the norms are equivalent, i.e. for each norm $|\,\cdot\,|_\ast$ there are positive constants $c_1,c_2$ such that for every $\vec x\in\R^d$ we have
\[
  c_1|\vec x|\leq|\vec x|_\ast\leq c_2|\vec x|,
\]
whenever we use the symbols $\gg$, $\ll$, $\asymp$, we assume that, generally, the implied constants depend on the norm $|\,\cdot\,|_\ast$\,, even though we do not point this out explicitly.

Let us set
\[
  \cH(\vec x)=\Big\{\vec y\in\R^d \,\Big|\, |\vec y|_\ast\leq|\vec x|_\ast,\ \Pi(\vec y)<\Pi(\vec x) \Big\}.
\]
It is obvious from Definition \ref{def:hyperbolic_minimum_arbitrary_norm} that a nonzero point $\vec x$ of $\La$ is a hyperbolic minimum of $\La$ w.r.t. $|\,\cdot\,|_\ast$ if and only if $\cH(\vec x)\cap\La=\{\vec 0\}$.

If $\vec x$ is a hyperbolic minimum of $\La$ (w.r.t. $|\,\cdot\,|_\ast$), then, of course, $-\vec x$ is also a hyperbolic minimum of $\La$. It may turn out that there are other hyperbolic minima with the same values of the functionals $|\cdot|_\ast$ and $\Pi(\,\cdot\,)$. Let us choose a representative out of each such set and let us order them by ascending norm $|\cdot|_\ast$. We get a sequence $\vec x_1,\vec x_2,\vec x_3,\ldots$ of hyperbolic minima (see Fig. \ref{fig:successive_hyperbolic_minima}), for which $|\vec x_k|_\ast<|\vec x_{k+1}|_\ast$, $\Pi(\vec x_{k+1})\leq\Pi(\vec x_k)$ and $\bar\cH(\vec x_k)\cap\La=\{\vec 0\}$, where
\[
  \bar\cH(\vec x_k)=\Big\{\vec y\in\R^d \,\Big|\, |\vec y|_\ast<|\vec x_{k+1}|_\ast,\ \Pi(\vec y)<\Pi(\vec x_k) \Big\}.
\]

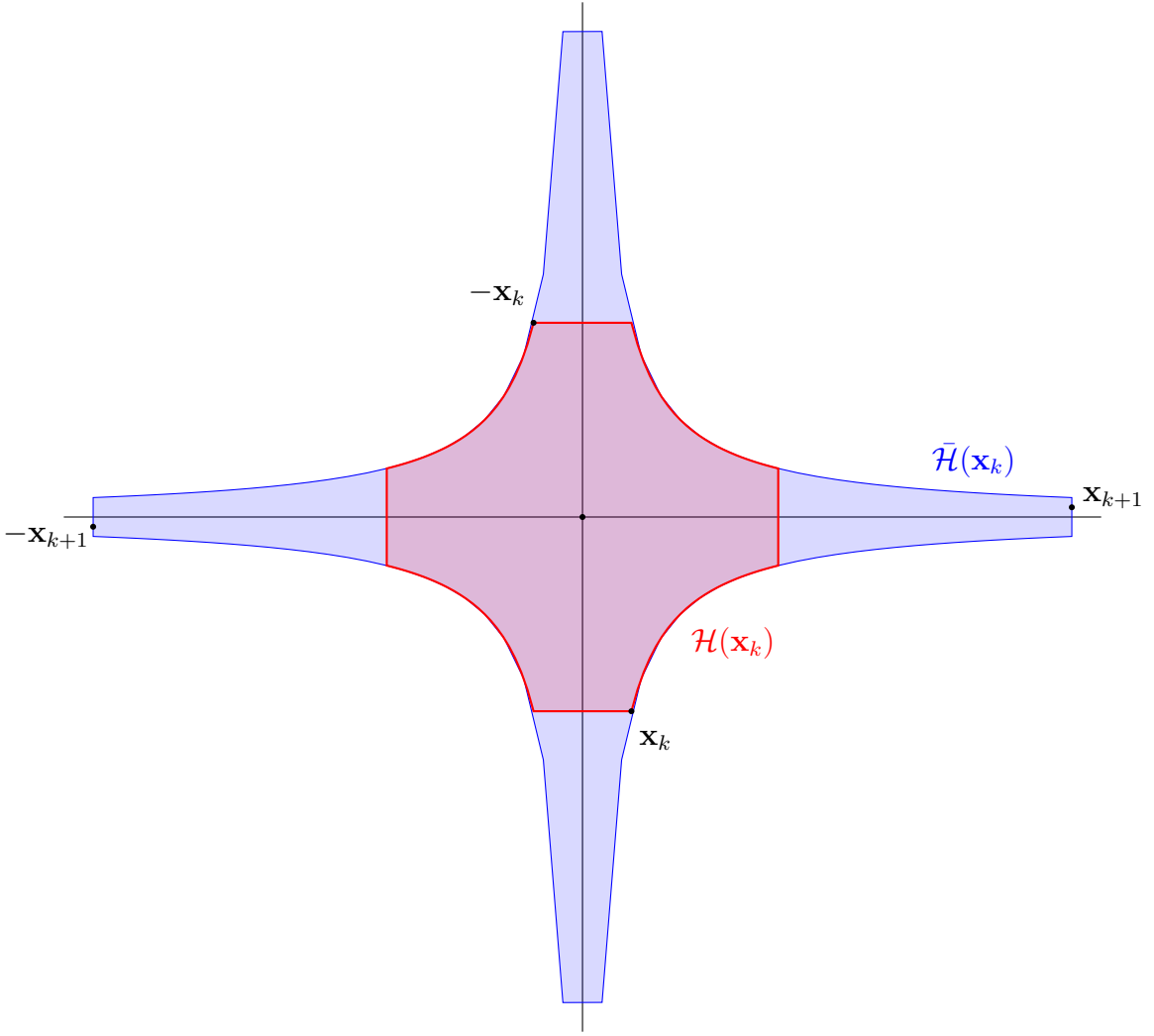
\begin{figure}[h]
\centering
\begin{tikzpicture}[domain=-5.1:5.1,scale=1.3]

  \draw (-5.3,0) -- (5.3,0); 
  \draw (0,-5.3) -- (0,5.3); 

  \fill[fill=blue,opacity=0.15]
      plot [domain=0.2:5] (\x,{1/(\x)}) --
      plot [domain=5:0.2] (\x, {-1/(\x)}) --
      plot [domain=-0.2:-5] (\x,{1/(\x)}) --
      plot [domain=-5:-0.2] (\x, {-1/(\x)}) --
      cycle;

  \fill[fill=red,opacity=0.15]
      plot [domain=0.5:2] (\x,{1/(\x)}) --
      plot [domain=2:0.5] (\x, {-1/(\x)}) --
      plot [domain=-0.5:-2] (\x,{1/(\x)}) --
      plot [domain=-2:-0.5] (\x, {-1/(\x)}) --
      cycle;

  \draw[color=blue]
      plot [domain=0.2:5] (\x,{1/(\x)}) --
      plot [domain=5:0.2] (\x, {-1/(\x)}) --
      plot [domain=-0.2:-5] (\x,{1/(\x)}) --
      plot [domain=-5:-0.2] (\x, {-1/(\x)}) --
      cycle;

  \draw[color=red, thick]
      plot [domain=0.5:2] (\x,{1/(\x)}) --
      plot [domain=2:0.5] (\x, {-1/(\x)}) --
      plot [domain=-0.5:-2] (\x,{1/(\x)}) --
      plot [domain=-2:-0.5] (\x, {-1/(\x)}) --
      cycle;

  \node[fill=black,circle,inner sep=0.8pt] at (0,0) {};
  \node[fill=black,circle,inner sep=0.8pt] at (-0.5,2) {};
  \node[fill=black,circle,inner sep=0.8pt] at (0.5,-2) {};
  \node[fill=black,circle,inner sep=0.8pt] at (5,0.1) {};
  \node[fill=black,circle,inner sep=0.8pt] at (-5,-0.1) {};
  
  \draw(0.47,-2.3) node[right]{$\vec x_k$};
  \draw(-0.47,2.3) node[left]{$-\vec x_k$};
  \draw(5,0.2) node[right]{$\vec x_{k+1}$};
  \draw(-4.93,-0.2) node[left]{$-\vec x_{k+1}$};

  \draw(4,0.3) node[above]{${\color[rgb]{0,0,1}\bar\cH(\vec x_k)}$};
  \draw(1,-1.3) node[right]{${\color[rgb]{1,0,0}\cH(\vec x_k)}$};

\end{tikzpicture}
\caption{Successive hyperbolic minima}\label{fig:successive_hyperbolic_minima}
\end{figure}

If $\Pi(\,\cdot\,)$ attains arbitrarily small values at nonzero points of $\La$ but does not vanish at any such point, the set of hyperbolic minima is infinite. In this case $\omega(\La)$, $\bar\omega(\La)$ satisfy the relations
\begin{equation}\label{eq:omega_in_terms_of_hyperbolic_minima}
  \omega(\La)=\sup\Big\{ \gamma\in\R \,\Big|\, \forall\,K\in\N\,\ \exists\,k\geq K:\,\Pi(\vec x_k)\leq|\vec x_k|_\ast^{-\gamma} \Big\},\ \
\end{equation} 
\begin{equation}\label{eq:bar_omega_in_terms_of_hyperbolic_minima}
  \bar\omega(\La)=\sup\Big\{ \gamma\in\R \,\Big|\, \exists\,K\in\N:\,\forall k\geq K\ \Pi(\vec x_k)\leq|\vec x_{k+1}|_\ast^{-\gamma} \Big\}.
\end{equation}

\section{Two-dimensional case and additional linear forms}\label{sec:2dim_and_additional_forms}

The two-dimensional case will play a key role in the proof of Theorem \ref{t:spectrum_of_the_weak_arbitrary_d}. More detailed, for $d\geq3$, we consider a two-dimensional subspace $\cL$ of general position, construct a rank $2$ lattice in it whose projection onto the plane of the first two coordinates has a given value of the regular or the uniform exponent, after which we complement this lattice to a full-rank lattice in $\R^d$. Respectively, since we must control the product of all the $d$ coordinates of points in $\cL$, we must learn how to work with the following modification of regular and uniform exponents in the two-dimensional case.

Let $\ell_1,\ldots,\ell_n$ be arbitrary nonzero linear forms in two variables, $n\geq1$. For brevity, we will denote this $n$-tuple by $\bL$, i.e. $\bL=(\ell_1,\ldots,\ell_n)$.

Let us set for each $\vec x=(x_1,x_2)\in\R^2$
\[
  \Pi_\bL(\vec x)=\bigg|x_1x_2\prod_{1\leq i\leq n}\ell_i(\vec x)\bigg|^{1/(n+2)}.
\]
Let us set for each rank $2$ lattice $\Ga$ in $\R^2$
\[
  \psi_{\Ga,\bL}(t)=\min_{\substack{ \vec x\in\Ga \\ 0<|\vec x|\leq t}}\Pi_\bL(\vec x)
\]
and define the exponents
\[
  \omega_\bL(\Ga)=\sup\Big\{ \gamma\in\R \,\Big|\, \liminf_{t\to+\infty}\big(t^\gamma\psi_{\Ga,\bL}(t)\big)<+\infty \Big\},
\]
\begin{equation}\label{eq:bar_omega_L_definition}
  \bar\omega_\bL(\Ga)=\sup\Big\{ \gamma\in\R \,\Big|\, \limsup_{t\to+\infty}\big(t^\gamma\psi_{\Ga,\bL}(t)\big)<+\infty \Big\}.
\end{equation}
Note that $\omega_\bL(\Ga)$, $\bar\omega_\bL(\Ga)$ are not necessarily nonnegative. The only inequalities that the definition and Minkowski's theorem provide are
\begin{equation}\label{eq:trivial_inequalities_for_omega_L}
  \omega_\bL(\Ga)\geq\bar\omega_\bL(\Ga)\geq-\frac n{n+2}\,.
\end{equation}

We say that $\bL$ is \emph{logarithmically badly approximable} w.r.t. $\Ga$ if there are constants $\e>0$ and $c>0$ such that for each $i=1,\ldots,n$ and each nonzero $\vec x\in\Ga$ we have
\[
  |\ell_i(\vec x)|>\frac c{|\vec x|\log^{1+\e}(1+|\vec x|)}.
\]

\begin{lemma}\label{l:from_omega_to_omega_L}
   Let $\Ga$ be a rank $2$ lattice in $\R^2$. Assume that all the coefficients of the linear forms in $\bL$ are nonzero and that no two forms in $\bL$ are proportional. Let $\bL$ be logarithmically badly approximable w.r.t. $\Ga$. Let
  \[
    \omega(\Ga)=\delta.
  \]
  Then
  \[
    \omega_\bL(\Ga)=
    \max\bigg(0,\,\dfrac{2\delta-n}{n+2}\bigg).
  \]
\end{lemma}

\begin{proof}
  Consider the sets
  \begin{equation}\label{eq:S_0}
    \cS_0=\Big\{ \vec x=(x_1,x_2)\in\R^2 \,\Big|\, \min\big(|x_1|,|x_2|\big)<1 \Big\},
  \end{equation}
  \begin{equation}\label{eq:S_j}
    \cS_j=\Big\{ \vec x\in\R^2 \,\Big|\, |\ell_j(\vec x)|<1 \Big\},\qquad\ \ j=1,\ldots,n.
  \end{equation}
  It follows from the conditions on the linear forms that for $\vec x\in\cS_0\backslash\{\vec 0\}$
  \[
    \Pi_\bL(\vec x)\asymp_\bL
    \Big(\Pi(\vec x)^2|\vec x|^n\Big)^{\frac1{n+2}}>
    |\vec x|^{-\frac{2\delta-n}{n+2}+o(1)}
    \ \ \text{ as }\ \ |\vec x|\to\infty,
  \]
  and for $\vec x\in\cS_j\backslash\{\vec 0\}$ with any $j=1,\ldots,n$
  \[
    \Pi_\bL(\vec x)\gg_\bL
    \bigg(\frac{|\vec x|^{n+1}}{|\vec x|\log^{1+\e}(1+|\vec x|)}\bigg)^{\frac1{n+2}}\gg_\e1.
  \]
  But if $\vec x\notin\cS_0\cup\cS_1\cup\ldots\cup\cS_n$, then $\Pi_\bL(\vec x)\geq1$. 
  
  According to \eqref{eq:omega_in_terms_of_hyperbolic_minima} there is a subsequence $(\vec x_{k_j})_{j\in\N}$ of hyperbolic minima of $\Ga$ w.r.t. the sup-norm for which
  \[
    \Pi(\vec x_{k_j})=|\vec x_{k_j}|^{-\delta+o(1)}
    \ \ \text{ as }\ \ j\to\infty.
  \]
  For these points we have
  \[
    \Pi_\bL(\vec x_{k_j})\asymp_\bL
    \Big(\Pi(\vec x_{k_j})^2|\vec x_{k_j}|^n\Big)^{\frac1{n+2}}=|\vec x_{k_j}|^{-\frac{2\delta-n}{n+2}+o(1)}
    \ \ \text{ as }\ \ j\to\infty.
  \]
  If $\delta<n/2$, then
  \[
    \frac{2\delta-n}{n+2}<0,
  \]
  hence in this case the functional $\Pi_\bL(\,\cdot\,)$ is bounded away from zero, which means that $\omega_\bL(\Ga)=0$. If $\delta\geq n/2$, we get
  \[
    \omega_\bL(\Ga)=\frac{2\delta-n}{n+2}\,.
  \]
  Lemma is proved.
\end{proof}

\begin{lemma}\label{l:from_bar_omega_to_bar_omega_L}
  Let $\Ga$ be a rank $2$ lattice in $\R^2$. Assume that all the coefficients of the linear forms in $\bL$ are nonzero and that no two forms in $\bL$ are proportional. Let $\bL$ be logarithmically badly approximable w.r.t. $\Ga$. Assume also that the sequence $(\vec x_k)_{k\in\N}$ of hyperbolic minima of $\Ga$ w.r.t. some norm $|\,\cdot\,|_\ast$ satisfies the relations
  \begin{equation}\label{eq:beta_growth_of_successive_hyperbolic_minima}
    |\vec x_{k+1}|_\ast\asymp|\vec x_k|_\ast^\beta\,,
    \qquad
    \Pi(\vec x_k)\asymp|\vec x_k|_\ast^{(1-\beta^2)/2},
    \qquad
    \beta\geq\sqrt{n+1}.
  \end{equation}
  Then
  \[
    \bar\omega_\bL(\Ga)=
    \dfrac{\beta-(n+1)\beta^{-1}}{n+2}\,.
  \]
\end{lemma}

\begin{proof}
  Define the sets $\cS_0,\cS_1,\ldots,\cS_n$ by \eqref{eq:S_0}, \eqref{eq:S_j}. As we showed in the proof of Lemma \ref{l:from_omega_to_omega_L}, we have $\Pi_\bL(\vec x)\gg_\bL1$ for $\vec x\notin\cS_0$. For $\vec x\in\cS_0\backslash\{\vec 0\}$, due to the fact that the coefficients of the linear forms are nonzero, we have
  \begin{equation}\label{eq:prod_ell_in_S_0}
    \prod_{j=1}^n|\ell_j(\vec x)|\asymp_\bL|\vec x|_\ast^n\,.
  \end{equation}
  Consider the sequence $(\vec x_k)_{k\in\N}$ of hyperbolic minima of $\Ga$ w.r.t. $|\,\cdot\,|_\ast$\,. By the definition of successive hyperbolic minima and relation \eqref{eq:prod_ell_in_S_0}, for each $k\in\N$ and every $\vec x\in\cS_0$ such that
  \[
    |\vec x_k|_\ast\leq|\vec x|_\ast<|\vec x_{k+1}|_\ast,
  \]
  we have
  \[
    \Pi_\bL(\vec x)\asymp_\bL
    \Big(\Pi(\vec x)^2|\vec x|_\ast^n\Big)^{\frac1{n+2}}\geq
    \Big(\Pi(\vec x_k)^2|\vec x_k|_\ast^n\Big)^{\frac1{n+2}}\asymp_\bL
    \Pi_\bL(\vec x_k).
  \]
  At the same time by \eqref{eq:beta_growth_of_successive_hyperbolic_minima} we have
  \begin{multline*}
    \Pi_\bL(\vec x_{k+1})
    \asymp_\bL
    \Big(\Pi(\vec x_{k+1})^2|\vec x_{k+1}|_\ast^n\Big)^{\frac1{n+2}}
    \asymp
    |\vec x_{k+1}|_\ast^{-\frac{\beta^2-(n+1)}{n+2}}
    \leq \\ \leq
    |\vec x_k|_\ast^{-\frac{\beta^2-(n+1)}{n+2}}
    \asymp
    \Big(\Pi(\vec x_k)^2|\vec x_k|_\ast^n\Big)^{\frac1{n+2}}
    \asymp_\bL
    \Pi_\bL(\vec x_k).
  \end{multline*}
  Therefore, the following analogue of \eqref{eq:bar_omega_in_terms_of_hyperbolic_minima} holds:
  \begin{equation}\label{eq:bar_omega_L_in_terms_of_hyperbolic_minima}
    \bar\omega_\bL(\Ga)=\sup\Big\{ \gamma\in\R \,\Big|\, \exists\,K\in\N:\,\forall k\geq K\ \Pi_\bL(\vec x_k)\leq|\vec x_{k+1}|_\ast^{-\gamma} \Big\}.
  \end{equation}
  Applying \eqref{eq:prod_ell_in_S_0} and \eqref{eq:beta_growth_of_successive_hyperbolic_minima} once again, we get
  \[
    \Pi_\bL(\vec x_k)\asymp_\bL
    \Big(\Pi(\vec x_k)^2|\vec x_k|_\ast^n\Big)^{\frac1{n+2}}\asymp
    |\vec x_k|_\ast^{-\frac{\beta^2-(n+1)}{n+2}}\asymp
    |\vec x_{k+1}|_\ast^{-\frac{\beta-(n+1)\beta^{-1}}{n+2}}.
  \]
  Hence, in view of \eqref{eq:bar_omega_L_in_terms_of_hyperbolic_minima}, it follows immediately that
  \[
    \bar\omega_\bL(\Ga)=
    \dfrac{\beta-(n+1)\beta^{-1}}{n+2}\,.
  \]
\end{proof}

\begin{remark*}
  In the proof of Lemma \ref{l:from_bar_omega_to_bar_omega_L} we actually showed that the hyperbolic minima of $\Ga$ play the role of analogues of hyperbolic minima (up to some bounded factors) for the functional $\Pi_\bL(\,\cdot\,)$. We can give a precise definition and call a nonzero point $\vec x\in\Ga$ an \emph{$\bL$-hyperbolic minimum} of $\Ga$ if there are no nonzero points $\vec y\in\Ga$ such that
  \[
    |\vec y|\leq|\vec x|
    \qquad\text{ and }\qquad
    \Pi_\bL(\vec y)<\Pi_\bL(\vec x).
  \]
  And of course, same as in the case of usual hyperbolic minima, we can define successive $\bL$-hyperbolic minima and formulate direct analogues of \eqref{eq:omega_in_terms_of_hyperbolic_minima}, \eqref{eq:bar_omega_in_terms_of_hyperbolic_minima} (see relation \eqref{eq:bar_omega_L_in_terms_of_hyperbolic_minima} in the proof of Lemma \ref{l:from_bar_omega_to_bar_omega_L}). 
\end{remark*}

\begin{lemma}\label{l:providing_log_bad}
  Let $\Ga$ be a rank $2$ lattice in $\R^2$. Assume that all the coefficients of the linear forms in $\bL$ are nonzero. Then, for almost every real $\tau$ the $n$-tuple $\bL$ is logarithmically badly approximable w.r.t. $D_\tau\Ga$, where
  \[
    D_\tau=
    \begin{pmatrix}
      2^\tau & 0\phantom{-} \\
      0 & 2^{-\tau}
    \end{pmatrix}.
  \]
\end{lemma}

\begin{proof}
  Let $\Ga=A\Z^2$,
  \[
    A=
    \begin{pmatrix}
      a_{1,1} & a_{1,2} \\
      a_{2,1} & a_{2,2}
    \end{pmatrix}
    \in\Gl(2,\R).
  \]
  For convenience, let us agree to view $\ell_i$ as column vectors. Then $\bL$ is logarithmically badly approximable w.r.t. $D_\tau\Ga=D_\tau A\Z^2$ if and only if the $n$-tuple
  \[
    A^\top D_\tau\bL=(A^\top D_\tau\ell_1,\ldots,A^\top D_\tau\ell_n)
  \]
  is logarithmically badly approximable w.r.t. $\Z^2$. Let us fix $i$ and assume that the coefficients of $\ell_i$ are equal to $a$ and $b$. Then
  \[
    A^\top D_\tau\ell_i=
    \begin{pmatrix}
      2^\tau a_{1,1} & 2^{-\tau}a_{2,1} \\
      2^\tau a_{1,2} & 2^{-\tau}a_{2,2} 
    \end{pmatrix}
    \binom ab=
    \begin{pmatrix}
      2^\tau a_{1,1}a + 2^{-\tau}a_{2,1}b \\
      2^\tau a_{1,2}a + 2^{-\tau}a_{2,2}b
    \end{pmatrix}=
    \begin{pmatrix}
      a(\tau) \\
      b(\tau)
    \end{pmatrix}.
  \]
  If $\tau$ runs over some interval that does not contain a zero of the function $b(\tau)$, the ratio $a(\tau)/b(\tau)$ also runs over some interval, since $ab\neq0$ and $a_{1,1}a_{2,2}-a_{1,2}a_{2,1}\neq0$. Further, for any fixed $\e>0$ the series
  \[
    \sum_{t=1}^{\infty}
    \frac 1{t\log^{1+\e}(1+t)}
  \]
  converges. Therefore, by Khintchine's theorem (see \cite{khintchine_CF}), given any $\e>0$, there is a constant $c=c(\e)>0$ such that for almost every $\tau\in\R$ the inequality
  \[
    |a(\tau)x_1+b(\tau)x_2|>\frac c{|\vec x|\log^{1+\e}(1+|\vec x|)}
  \]
  holds for every nonzero $\vec x=(x_1,x_2)\in\Z^2$. Taking into account that the intersection of finitely many sets of full measure is also a set of full measure, we get the desired statement.
\end{proof}

\section{Constructing a two-dimensional subspace and a rank $2$ lattice in it}\label{sec:rank_2_sublattice_construction}

Let $d\in\N$, $d\geq3$. Set $n=d-2$.

Take arbitrary real numbers $a_1,\ldots,a_n$, $b_1,\ldots,b_n$ such that
\begin{equation}\label{eq:bounds_for_a_i_and_b_i}
  0<a_i<1/2,
  \qquad
  0<b_i<1/2,
  \qquad
  i=1,\ldots,n.
\end{equation}
Define vectors $\vec a,\vec b\in\R^d$ and a matrix $L$ by
\[
  \vec a=(1,0,a_1,\ldots,a_n),
  \qquad
  \vec b=(0,1,b_1,\ldots,b_n),
  \qquad
  L=
  \begin{pmatrix}
    1 & 0 \\
    0 & 1 \\
    a_1 & b_1 \\
    \vdots & \vdots \\
    a_n & b_n
  \end{pmatrix}.
\]
Consider the subspace
\[
  \cL=L\R^2=\R\vec a+\R\vec b
\]
of $\R^d$. For each lattice $\Ga=A\Z^2$, $A\in\Gl(2,\R)$, set
\[
  \Ga_\cL=L\Ga=
  \begin{pmatrix}
    1 & 0 \\
    0 & 1 \\
    a_1 & b_1 \\
    \vdots & \vdots \\
    a_n & b_n
  \end{pmatrix}
  A\Z^2.
\]
Then $\Ga_\cL$ is a rank $2$ lattice in $\cL$ and all the points of $\Ga$ are obtained from the points of $\Ga_\cL$ by deleting the last $n$ coordinates.

Let us define $\bL=(\ell_1,\ldots,\ell_n)$ as follows. For each $i=1,\ldots,n$ and every $\vec x=(x_1,x_2)\in\R^2$, let us set
\[
  \ell_i(\vec x)=a_ix_1+b_ix_2.
\]

\begin{lemma}\label{l:lifting_up_the_functionals}
  Given a point $\vec x=(x_1,x_2)\in\R^2$, consider the point $L\vec x=x_1\vec a+x_2\vec b$ in $\cL$. Then
  \[
    |L\vec x|=|\vec x|
    \qquad\text{ and }\qquad
    \Pi(L\vec x)=\Pi_\bL(\vec x).
  \]
\end{lemma}

\begin{proof}
  In the vector notation we have
  \[
    x_1\vec a+x_2\vec b=
    \begin{pmatrix}
      x_1 \\
      x_2 \\
      a_1x_1+b_1x_2 \\
      \vdots \\
      a_nx_1+b_nx_2
    \end{pmatrix}=
    \begin{pmatrix}
      x_1 \\
      x_2 \\
      \ell_1(\vec x) \\
      \vdots \\
      \ell_n(\vec x)
    \end{pmatrix}.
  \]
  This implies immediately that
  \[
    \Pi(L\vec x)=
    \bigg|x_1x_2\prod_{1\leq i\leq n}\ell_i(\vec x)\bigg|^{1/d}=
    \Pi_\bL(\vec x).
  \]
  Taking into account \eqref{eq:bounds_for_a_i_and_b_i}, we get
  \[
    |\ell_i(\vec x)|=
    |a_ix_1+b_ix_2|\leq
    \frac12|x_1|+\frac12|x_2|\leq
    \max(|x_1|,|x_2|).
  \]
  Hence $\max\big(|x_1|,|x_2|,|\ell_1(\vec x)|,\ldots,|\ell_n(\vec x)|\big)=\max(|x_1|,|x_2|)=|\vec x|$.
\end{proof}

\section{Complementing $\Ga_\cL$ to a lattice of rank $d$}\label{sec:complement_to_La}

Let us choose a basis in $\Ga_\cL$ consisting of vectors with positive coordinates. This is possible due to the choice of $\vec a$ and $\vec b$. Denote the chosen basis vectors by $\vec v_1$, $\vec v_2$. Let vectors $\vec e_1,\ldots,\vec e_n$ complement $\vec v_1,\vec v_2$ to a basis of a lattice $\La$ of rank $d$. Then the vectors $\vec e_1,\ldots,\vec e_n,\vec v_1,\vec v_2$ are linearly independent and
\[
  \La=\Z\vec e_1+\ldots+\Z\vec e_n+\Z\vec v_1+\Z\vec v_2.
\]

\begin{lemma}\label{l:2dim_provides_all}
  Define the subspace $\cL$ and the lattices $\Ga$, $\Ga_\cL$ as in Section \ref{sec:rank_2_sublattice_construction}. Assume that there are constants $\e>0$ and $c>0$ such that for each $\vec v\in\La\backslash\Ga_\cL$ we have
  \begin{equation}\label{eq:2dim_provides_all_condition}
    \Pi(\vec v)>\frac c{\log^{1+\e}(1+|\vec v|)}\,.
  \end{equation}
  Then
  \[
    \omega(\La)=\max(0,\omega_\bL(\Ga))
    \qquad\text{ and }\qquad
    \bar\omega(\La)=\max(0,\bar\omega_\bL(\Ga)).
  \]
\end{lemma}

\begin{proof}
  If $\omega_\bL(\Ga)\leq0$ (see \eqref{eq:trivial_inequalities_for_omega_L}), then by Lemma \ref{l:lifting_up_the_functionals} and condition \eqref{eq:2dim_provides_all_condition}, for each nonzero $\vec v\in\La$, we have
  \[
    \Pi(\vec v)>|\vec v|^{o(1)}
    \ \ \text{ as }\ \ |\vec v|\to\infty.
  \]
  In this case we get $\omega(\La)=\bar\omega(\La)=0$.
  
  Let $\omega_\bL(\Ga)=\delta>0$. For $\vec v\in\La\backslash\Ga_\cL$, condition \eqref{eq:2dim_provides_all_condition} implies the estimate
  \[
    \Pi(\vec v)\gg|\vec v|^{-\delta/3}.
  \]
  Thus, every $\vec v\in\La$ with $|\vec v|$ large enough satisfying the inequality
  \[
    \Pi(\vec v)<|\vec v|^{-2\delta/3}
  \]
  belongs to $\Ga_\cL$\,. Taking into account Lemma \ref{l:lifting_up_the_functionals}, we get that for $\gamma>2\delta/3$
  \[
    \liminf_{t\to+\infty}\big(t^\gamma\psi_{\La}(t)\big)=
    \liminf_{t\to+\infty}\big(t^\gamma\psi_{\Ga,\bL}(t)\big),
  \]
  whence it follows immediately that $\omega(\La)=\omega_\bL(\Ga)$.
  
  Now, let $\bar\omega_\bL(\Ga)=\delta>0$. As we noted previously, for $\vec v\in\La\backslash\Ga_\cL$, we have the lower estimate
  \[
    \Pi(\vec v)\gg|\vec v|^{-\delta/3},
  \]
  whereas for the function $\psi_{\La,\bL}(t)$, we have the upper estimate
  \[
    \psi_{\La,\bL}(t)\ll t^{-2\delta/3},
  \]
  due to the defining equality \eqref{eq:bar_omega_L_definition}. Taking into account Definition \ref{def:lattice_exponents_liminf_limsup}, relation \eqref{eq:bar_omega_L_definition}, and Lemma \ref{l:lifting_up_the_functionals}, we get
  \[
    \psi_{\La}(t)=
    \min_{\substack{ \vec v\in\La \\ 0<|\vec v|\leq t}}\Pi(\vec v)\leq
    \min_{\substack{ \vec v\in\Ga_\cL \\ 0<|\vec v|\leq t}}\Pi(\vec v)=
    \min_{\substack{ \vec x\in\Ga \\ 0<|\vec x|\leq t}}\Pi_\bL(\vec x)=
    \psi_{\Ga,\bL}(t)\ll t^{-2\delta/3}.
  \]
  Therefore, we have
  \[
    \psi_{\La}(t)=\psi_{\Ga,\bL}(t)
  \]
  for every $t$ large enough, whence it follows immediately that $\bar\omega(\La)=\bar\omega_\bL(\Ga)$.
\end{proof}

\section{Main metric lemma}\label{sec:metric_lemma}

Let the points $\vec v_1,\vec v_2$ and the lattice $\Ga_\cL=\Z\vec v_1+\Z\vec v_2$ be as in Section \ref{sec:complement_to_La}. Let, as before, $n=d-2$. Given points $\vec b_1,\ldots,\vec b_n\in\R^d$, $\vec b_i=(b_{i,1},\ldots,b_{i,d})$, $i=1,\ldots,n$, consider unit cubes
\[
  \cB_i=\Big\{\vec x=(x_1,\ldots,x_d)\in\R^d \,\Big|\, b_{i,j}\leq|x_j|\leq b_{i,j}+1,\ j=1,\ldots,d \Big\}.
\]
Set also $\bB=\cB_1\times\ldots\times\cB_n$. Then $\bB$ is also a unit cube in $\R^{nd}$. For each choice of $\vec e_1\in\cB_1,\ \ldots,\ \vec e_n\in\cB_n$ the $n$-tuple $\bE=(\vec e_1,\ldots,\vec e_n)$ satisfies
\[
  \bE\in\bB.
\]

\begin{lemma}\label{l:main_metric_argument}
  Let the cube\,\ $\bB$ be chosen so that for each $\bE=(\vec e_1,\ldots,\vec e_n)\in\bB$ the vectors $\vec v_1,\vec v_2,\vec e_1,\ldots,\vec e_n$ are linearly independent. Let us set for each $\bE\in\bB$
  \[
    \La_\bE=\Z\vec e_1+\ldots+\Z\vec e_n+\Z\vec v_1+\Z\vec v_2.
  \]
  Then for every $\e>0$ there is a constant $c=c(\e)>0$ such that for almost every $\bE\in\bB$ the inequality
  \[
    \Pi(\vec x)>\frac c{\log^{1+\e}(1+|\vec x|)}
  \]
  holds for every $\vec x\in\La_\bE\backslash\Ga_\cL$.
\end{lemma}

\begin{proof}
  Let us fix $\e>0$. For all positive $t$, set
  \[
    \varphi(t)=\frac1{\log^{d+\e}(1+t)}\,.
  \]
  Let us define the ``forbidden'' set
  \[
    \cF=\cF(\varphi)=\Big\{\vec x\in\R^d \,\Big|\, \Pi(\vec x)^d\leq\varphi(|\vec x|),\ |\vec x|\geq1 \Big\}
  \]
  and consider the ``exceptional'' set
  \[
    \cM=\Big\{\bE\in\bB \,\Big|\, \text{the set }\big(\La_\bE\backslash\Ga_\cL\big)\cap\cF\text{ is infinite} \Big\}.
  \]
  Let us show that $\mu(\cM)=0$, where $\mu$ is the Lebesgue measure on the space $\R^{nd}$. We split our argument into several steps.
  
  \paragraph{Reduction to Borel--Cantelli lemma.}
  
  Each vector $\vec x\in\La_\bE$ enjoys a representation
  \begin{equation}\label{eq:a_point_of_La_E}
    \vec x=z_1\vec e_1+\ldots+z_n\vec e_n+z_{n+1}\vec v_1+z_{n+2}\vec v_2
  \end{equation}
  with integer $z_1,\ldots,z_{n+2}$ (recall that $n=d-2$). And conversely, for a fixed $\bE$, each $\vec z=(z_1,\ldots,z_d)\in\Z^d$ determines a vector $\vec x=\vec x(\vec z,\bE)$ of the form \eqref{eq:a_point_of_La_E}. Moreover,
  \[
    \vec x(\vec z,\bE)\in\Ga_\cL
    \iff
    \vec z\in\cV,
  \]
  where $\cV$ is the coordinate subspace of the last two coordinates.

  For each $\nu\in\Z$, $\nu\geq0$, let us consider the set
  \[
    \cM_\nu=\Big\{\bE\in\bB \,\Big|\, \exists\,\vec z\in\Z^d\backslash\cV:\,2^\nu\leq|\vec z|<2^{\nu+1},\ \vec x(\vec z,\bE)\in\cF \Big\}.
  \]
  Then
  \[
    \cM=\bigcap_{\nu_0=0}^\infty\bigcup_{\nu=\nu_0}^\infty\cM_\nu.
  \]
  By Borel--Cantelli lemma it suffices to prove the convergence of the series
  \begin{equation}\label{eq:borel_cantelli_series}
    \sum_{\nu=0}^\infty\mu(\cM_\nu).
  \end{equation}
  
  \paragraph{Sets $\cF^i_\nu$ and $\cM^i_\nu$.}  
  
  Note that if
  \[
    2^\nu\leq|\vec z|<2^{\nu+1},
  \]
  then by compactness of $\bB$ there is a positive integer $\eta$ depending only on $\bB$ such that
  \[
    2^{\nu-\eta}\leq|\vec x(\vec z,\bE)|\leq2^{\nu+\eta}.
  \]
  Set
  \[
    \cF_\nu=\Big\{\vec x\in\R^d \,\Big|\, \Pi(\vec x)^d\leq\varphi(|\vec x|),\ 2^{\nu-\eta}\leq|\vec x|\leq2^{\nu+\eta} \Big\}.
  \]
  Set also, for each $i=1,\ldots,d$,
  \[
    \cF^i_\nu=\Big\{\vec x\in\cF_\nu \,\Big|\, |\vec x|=x_i \Big\}.
  \]
  Then
  \[
    \cM_\nu\subseteq\bigcup_{i=1}^d\cM^i_\nu,
  \]
  where
  \[
    \cM^i_\nu=\Big\{\bE\in\bB \,\Big|\, \exists\,\vec z\in\Z^d\backslash\cV:\,2^\nu\leq|\vec z|<2^{\nu+1},\ \vec x(\vec z,\bE)\in\cF^i_\nu \Big\}.
  \]
  
  \paragraph{Covering $\cF^1_\nu$ by parallelepipeds.}
  
  Let us cover the set $\cF^1_\nu$ by a family of paralle\-le\-pi\-peds. This will allow us to estimate the measure of $\cM^1_\nu$\,. The measure of all the other $\cM^i_\nu$ is estimated in the same way.

  For each $\vec x=(x_1,x_2,\ldots,x_d)\in\R^d$, set $\underline{\vec x}=(x_2,\ldots,x_d)\in\R^{d-1}$.
  For each $\tau\in\Z$, $\tau\geq0$, define the $(d-1)$-dimensional set
  \[
    \cH_\tau=\Big\{\underline{\vec x}=(x_2,\ldots,x_d)\in\R^{d-1} \,\Big|\, |\underline{\vec x}|\leq2^{\tau+1},\ 
                                                                            |x_2\cdot\ldots\cdot x_d|\leq\frac{\varphi(2^\tau)}{2^\tau} \Big\}.
  \]
  Then 
  \begin{equation}\label{eq:F_1_nu_hyperbolic_covering}
    \cF^1_\nu\subset\bigcup_{\tau=\nu-\eta}^{\nu+\eta}\Big([2^\tau,2^{\tau+1}]\times\cH_\tau\Big).
  \end{equation}

  Let us cover the set $\cH_\tau$ by parallelepipeds as follows. Define an integer $\sigma=\sigma(\tau)$ by the inequalities
  \[
    2^{-\sigma-1}<
    \bigg(\dfrac{\varphi(2^\tau)}{2^\tau}\bigg)^{\frac1{d-1}}\leq
    2^{-\sigma}.
  \]
  Then
  \begin{equation}\label{eq:sigma_of_tau}
    \sigma=\sigma(\tau)=\bigg[\frac{\tau-\log_2(\varphi(2^\tau))}{d-1}\bigg].
  \end{equation}
  Consider the family of $(d-1)$-tuples of integers
  \[
    \cT=\cT(\tau)=\Big\{\underline{\pmb\tau}=(\tau_2,\ldots,\tau_d)\in\Z^{d-1} \,\Big|\, \tau_2+\ldots+\tau_d=0,\ |\underline{\pmb\tau}|\leq(d-1)(\sigma+\tau+2) \Big\}
  \]  
  and determine for each $\underline{\pmb\tau}\in\cT$ the parallelepiped
  \[
    \cP(\underline{\pmb\tau})=
    \Big\{\underline{\vec x}=(x_2,\ldots,x_d)\in\R^{d-1} \,\Big|\, |x_i|\leq2^{\tau_i-\sigma+(d-2)},\ i=2,\ldots,d \Big\}.
  \]
  Let us show that
  \begin{equation}\label{eq:H_tau_covering}
    \cH_\tau\subset\bigcup_{\underline{\pmb\tau}\in\cT}\cP(\underline{\pmb\tau}).
  \end{equation}
  Indeed, let $\underline{\vec y}\in\cH_\tau$\,, $|y_2\cdot\ldots\cdot y_d|=\dfrac{\varphi(2^\tau)}{2^\tau}$\,. Set
  \[
    \underline{\vec w}=(w_2,\ldots,w_d)=
    \underline{\vec y}\cdot2^{-\sigma}\cdot\bigg(\dfrac{\varphi(2^\tau)}{2^\tau}\bigg)^{-\frac1{d-1}}.
  \]
  The points $\underline{\vec y}$ and $\underline{\vec w}$ are proportional,
  \begin{equation}\label{eq:bounds_for_w}
    |\underline{\vec y}|\leq|\underline{\vec w}|<2|\underline{\vec y}|\leq2^{\tau+2},
    \qquad\text{ and }\qquad
    |w_2\cdot\ldots\cdot w_d|=2^{-(d-1)\sigma}.
  \end{equation}
  Set $\tau_i=\sigma+\big\lceil\log_2|w_i|\big\rceil$, $i=2,\ldots,d-1$ and $\tau_d=-\tau_2-\ldots-\tau_{d-1}$. Here $\lceil\,\cdot\,\rceil$ denotes the ceiling function. Then by \eqref{eq:bounds_for_w}
  \[
    \quad
    2^{\tau_i-\sigma-1}<|w_i|\leq2^{\tau_i-\sigma},
    \qquad\qquad\ \
    |\tau_i|\leq\sigma+\tau+2,
    \qquad
    i=1,\ldots,d-1,
  \]
  \[
    |w_d|<
    2^{\tau_d-\sigma+(d-2)},
    \qquad
    |\tau_d|\leq(d-1)(\sigma+\tau+2).
  \]
  Thus, $\underline{\vec w}$ and, therefore, $\underline{\vec y}$ belong to the parallelepiped $\cP(\underline{\pmb\tau})$. I.e., indeed, inclusion \eqref{eq:H_tau_covering} takes place.
  
  Taking into account \eqref{eq:F_1_nu_hyperbolic_covering}, we get the following covering of $\cF^1_\nu$ by parallelepipeds:
  \[
    \cF^1_\nu=
    \bigcup_{\substack{\pmb\tau=(\tau_1,\ldots,\tau_d)\,\in\,\Z^d \\ 
             \nu-\eta\,\leq\,\tau_1\,\leq\,\nu+\eta \vphantom{\frac11} \\ 
             \underline{\pmb\tau}\,\in\,\cT(\tau_1)}}
    \cQ(\pmb\tau),
    \qquad\text{ where }\qquad
    \cQ(\pmb\tau)=
    [2^{\tau_1},2^{\tau_1+1}]\times\cP(\underline{\pmb\tau}).
  \]
  
  This covering in turn gives the following covering of the set $\cM^1_\nu$:
  \[
    \cM^1_\nu=
    \bigcup_{\substack{\pmb\tau=(\tau_1,\ldots,\tau_d)\,\in\,\Z^d \\ 
             \nu-\eta\,\leq\,\tau_1\,\leq\,\nu+\eta \vphantom{\frac11} \\ 
             \underline{\pmb\tau}\,\in\,\cT(\tau_1)}}
    \bigcup_{\substack{\vec z\in\Z^d\backslash\cV \\ 
             \,\ 2^\nu\leq|\vec z|<2^{\nu+1}}}
    \cN(\vec z,\pmb\tau),
  \]
  where
  \[
    \cN(\vec z,\pmb\tau)=
    \Big\{\bE\in\bB \,\Big|\, \vec x(\vec z,\bE)\in\cQ(\pmb\tau) \Big\}.
  \]
  Thus, the measure of $\cM^1_\nu$ can be estimated as
  \begin{equation}\label{eq:estimate_by_measure_of_N}
    \mu(\cM^1_\nu)\leq
    \sum_{\tau_1=\nu-\eta}^{\nu+\eta}\phantom{1}
    \sum_{\underline{\pmb\tau}\,\in\,\cT(\tau_1)}
    \sum_{\substack{\vec z\in\Z^d\backslash\cV \\ 
                    \,\ 2^\nu\leq|\vec z|<2^{\nu+1}}}
    \mu(\cN(\vec z,\pmb\tau)).
  \end{equation}
  
  \paragraph{\bf Analysis of $\cN(\vec z,\pmb\tau)$}
  
  Let us write the coordinates of $\vec e_1,\ldots,\vec e_n$ and $\vec v_1,\vec v_2$ in the columns of matrices $E$ and $V$, respectively:
  \[
    E=
    \begin{pmatrix}
      e_{1,1} & \cdots & e_{n,1} \\
      \vdots  & \ddots & \vdots \\
      e_{1,d} & \cdots & e_{n,d}
    \end{pmatrix},
    \qquad
    V=
    \begin{pmatrix}
      v_{1,1} & v_{2,1} \\
      \vdots  & \vdots \\
      v_{1,d} & v_{2,d}
    \end{pmatrix}.
  \]
  Denote by $\vec e^1,\ldots,\vec e^d$ the rows of $E$ and by $\vec v^1,\ldots,\vec v^d$ -- the rows of $V$. Let us set for $\vec z=(z_1,\ldots,z_d)$
  \[
    \vec z_E=(z_1,\ldots,z_n)
    \qquad\text{ and }\qquad
    \vec z_V=(z_{n+1},z_{n+2}).
  \]
  Let us denote by $\langle\,\cdot\,,\,\cdot\,\rangle$ the inner product. Then the condition $\vec x(\vec z,\bE)\in\cQ(\pmb\tau)$ is equivalent to the system of inequalities
  \begin{equation}\label{eq:forbidden_set_system}
    \begin{cases}
      2^{\tau_1}\leq\langle\vec e^1,\vec z_E\rangle+\langle\vec v^1,\vec z_V\rangle\leq2^{\tau_1+1} \\
      |\langle\vec e^2,\vec z_E\rangle+\langle\vec v^2,\vec z_V\rangle|\leq2^{\tau_2-\sigma+(d-2)} \\
      \ \cdot\cdot\cdot\cdot\cdot\cdot\cdot\cdot\cdot\cdot\cdot\cdot\cdot\cdot\cdot\cdot\cdot\cdot\cdot\cdot\cdot\cdot\cdot\cdot\cdot\cdot\cdot \\
      |\langle\vec e^d,\vec z_E\rangle+\langle\vec v^d,\vec z_V\rangle|\leq2^{\tau_d-\sigma+(d-2)}
    \end{cases},
  \end{equation}
  where $\sigma=\sigma(\tau_1)$ is defined by \eqref{eq:sigma_of_tau}.
  
  For every $j=1,\ldots,d$, the set of all possible values of $\vec e^j$, i.e. the set
  \[
    \Big\{\vec e^j \,\Big|\, \bE\in\bB \Big\},
  \]
  is a unit cube in $\R^n$. Denote this cube by $\cB^j$. Then
  \[
    \cN(\vec z,\pmb\tau)=
    \cN^1(\vec z,\pmb\tau)\times\ldots\times\cN^d(\vec z,\pmb\tau),
  \]
  where
  \[
    \cN^1(\vec z,\pmb\tau)=
    \Big\{\vec e^1\in\cB^1 \,\Big|\, 2^{\tau_1}\leq\langle\vec e^1,\vec z_E\rangle+\langle\vec v^1,\vec z_V\rangle\leq2^{\tau_1+1} \Big\}
  \]
  and
  \[
    \cN^j(\vec z,\pmb\tau)=
    \Big\{\vec e^j\in\cB^j \,\Big|\, |\langle\vec e^j,\vec z_E\rangle+\langle\vec v^j,\vec z_V\rangle|\leq2^{\tau_j-\sigma+(d-2)} \Big\}
  \]
  for $j=2,\ldots,d$.
  
  We estimate roughly the measure of $\cN^1(\vec z,\pmb\tau)$:
  \[
    \mu(\cN^1(\vec z,\pmb\tau))\leq1.
  \]
  
  The set $\cN^j(\vec z,\pmb\tau)$ for a given $j=2,\ldots,d$, in general, can be empty. If it is not empty, its measure can be obviously estimated as
  \[
    \mu(\cN^j(\vec z,\pmb\tau))\ll
    \frac{2^{\tau_j-\sigma+(d-2)}}{|\vec z_E|}\ll_d
    \frac{2^{\tau_j-\sigma}}{|\vec z_E|}\,.
  \]
  Since $\tau_2+\ldots+\tau_d=0$ and $\sigma=\sigma(\tau_1)$ is defined by \eqref{eq:sigma_of_tau}, we get
  \begin{equation}\label{eq:estimating_measure_of_N}
    \mu(\cN(\vec z,\pmb\tau))\ll_d
    \frac{2^{-(d-1)\sigma}}{|\vec z_E|^{d-1}}\ll_d
    \frac{2^{-\tau_1+\log_2(\varphi(2^{\tau_1}))}}{|\vec z_E|^{d-1}}=
    \frac{2^{-\tau_1}\cdot\varphi(2^{\tau_1})}{|\vec z_E|^{d-1}}\,.
  \end{equation}
  Note that this estimate is uniform in $\underline{\pmb\tau}$\,. And it is nontrivial only in the case when each of the sets $\cN^j(\vec z,\pmb\tau)$, $j=2,\ldots,d$, is not empty.
  
  \paragraph{Estimating the number of nonempty $\cN^j(\vec z,\pmb\tau)$ for a fixed $\vec z_E$.}
  
  Recall that the vectors $\vec v^j$, $j=2,\ldots,d$, are fixed and distinct from the zero vector (as the coordinates of $\vec v_1$, $\vec v_2$ are positive). Let us fix a vector $\vec z_E$, an index $j$, $2\leq j\leq d$, and a positive constant $C$. Let us ask ourselves what is the cardinality of the set
  \[
    \cR^j(\vec z_E,C)=
    \Big\{\vec z_V\in\Z^2 \,\Big|\, |\vec z_V|\leq C,\ \exists\,\vec e\in\cB^j:\ |\langle\vec e,\vec z_E\rangle+\langle\vec v^j,\vec z_V\rangle|\leq1 \Big\}.
  \]
  The set of values of $\langle\vec e,\vec z_E\rangle$ as $\vec e$ ranges through the cube $\cB^j$ is a segment of length of order $|\vec z_E|$. Hence the set $\cR^j(\vec z_E,C)$ is contained in a rectangle with sides of order $C$ and $|\vec z_E|$. The number of integer points in such a rectangle is bounded by a quantity of the order of its area, i.e.
  \[
    |\cR^j(\vec z_E,C)|\ll C|\vec z_E|.
  \]
  Finally, we note that since
  \[
    \sum_{j=2}^{d}\big(\tau_j-\sigma+(d-2)\big)=-(d-1)\sigma+O(d)<-\tau_1+\log_2(\varphi(2^{\tau_1}))+O(d),
  \]
  and $\tau_1$ is large, there is $j$ such that $2^{\tau_j-\sigma+(d-2)}<1$. Hence, for the set
  \[
    \cR(\vec z_E,C)=
    \bigcap_{j=2}^d\cR^j(\vec z_E,C),
  \]
  we have
  \begin{equation}\label{eq:estimate_on_R}
    |\cR(\vec z_E,C)|\ll_\bB C|\vec z_E|.
  \end{equation}
  
  \paragraph{Estimating the measure of $\cM_\nu$.}
  
  By \eqref{eq:estimate_by_measure_of_N}, \eqref{eq:estimating_measure_of_N}, \eqref{eq:estimate_on_R} we get the following estimate for the measure of $\cM^1_\nu$\,:
  \begin{align*}
    \mu(\cM^1_\nu)\, & \leq
    \sum_{\tau_1=\nu-\eta}^{\nu+\eta}\phantom{1}
    \sum_{\underline{\pmb\tau}\,\in\,\cT(\tau_1)}
    \sum_{\substack{\vec z\in\Z^d\backslash\cV \\ 
                    \,\ 2^\nu\leq|\vec z|<2^{\nu+1}}}
    \mu(\cN(\vec z,\pmb\tau))
    \ll_\bB \\ & \ll_\bB
    \sum_{\tau_1=\nu-\eta}^{\nu+\eta}\phantom{1}
    \sum_{\underline{\pmb\tau}\,\in\,\cT(\tau_1)}
    \sum_{\substack{\vec z_E\in\Z^n \\ 
                    \,\ 0<|\vec z_E|<2^{\nu+1}}}
    \sum_{\vec z_V\in\cR(\vec z_E,2^{\nu+1})}
    \frac{2^{-\tau_1}\cdot\varphi(2^{\tau_1})}{|\vec z_E|^{d-1}}
    \ll_\bB \\ & \ll_\bB
    \sum_{\tau_1=\nu-\eta}^{\nu+\eta}
    |\cT(\tau_1)|\cdot2^{-\tau_1}\cdot\varphi(2^{\tau_1})
    \hskip-5mm
    \sum_{\substack{\vec z_E\in\Z^n \\ 
                    \,\ 0<|\vec z_E|<2^{\nu+1}}}
    \frac{2^{\nu+1}}{|\vec z_E|^{d-2}}
    \ll_\bB \\ & \ll_\bB
    \sum_{\tau_1=\nu-\eta}^{\nu+\eta}
    \big(d\tau_1-\log_2(\varphi(2^{\tau_1})\big)^{d-2}\cdot2^{-\tau_1+\nu+1}\cdot\varphi(2^{\tau_1})
    \sum_{T=1}^{2^{\nu+1}}
    \sum_{\substack{\vec z_E\in\Z^n \\ 
                    |\vec z_E|=T}}
    \frac1{|\vec z_E|^{d-2}}
    \ll_\bB \\ & \ll_\bB
    \Big(\nu+\big|\log_2\big(\varphi(2^{\nu+\eta})\big)\big|\Big)^{d-2}\hskip-1mm\cdot\varphi(2^{\nu-\eta})
    \sum_{T=1}^{2^{\nu+1}}\frac1T
    \ll_\bB \\ & \ll_\bB
    \ \nu\cdot\varphi(2^{\nu-\eta})\cdot\Big(\nu^{d-2}+\big|\log_2(\varphi\big(2^{\nu+\eta})\big)\big|^{d-2}\Big).
    \vphantom{\frac{\Big|}{}}
  \end{align*}
  
  The measure of the sets $\cM^2_\nu,\ldots,\cM^d_\nu$ is estimated in the same way. Hence
  \[
    \mu(\cM_\nu)\,\ll_\bB\,
    \varphi(2^{\nu-\eta})\cdot\Big(\nu^{d-1}+\nu\big|\log_2(\varphi\big(2^{\nu+\eta})\big)\big|^{d-2}\Big),
  \]
  where $\eta$ is a positive constant depending on $\bB$. Taking into account that
  \[
    \varphi(t)=\frac1{\log^{d+\e}(1+t)}\,,
  \]
  we get
  \[
    \mu(\cM_\nu)\ll_\bB\nu^{-1-\e},
  \]
  i.e. the series \eqref{eq:borel_cantelli_series} converges. It follows by Borel--Cantelli lemma that the measure of the ``exceptional'' set $\cM$ is equal to zero. Thus, for almost every $\bE\in\bB$ the inequality $\Pi(\vec x)^d\leq\varphi(|\vec x|)$ holds only for a finite number of points of the set $\La_\bE\backslash\Ga_\cL$. Which immediately implies the desired statement.
\end{proof}

\section{Proof of Theorems \ref{t:spectrum_of_the_regular_arbitrary_d} and \ref{t:spectrum_of_the_weak_arbitrary_d}}\label{sec:proofs}

We prove both theorems using a common scheme. After choosing numbers $\theta,\eta\in\R$ appropriately, we consider the lattice
\[
  \Ga_{\theta,\eta}=
  \begin{pmatrix}
    \theta & -1 \\
    1 & \phantom{-}\eta
  \end{pmatrix}
  \Z^2.
\]
We construct $\bL=(\ell_1,\ldots,\ell_n)$ as in Section \ref{sec:rank_2_sublattice_construction}. By Lemma \ref{l:providing_log_bad} there exists $\tau\in\R$ such that $\bL$ is logarithmically badly approximable w.r.t. the lattice $D_\tau\Ga_{\theta,\eta}$. We take such a $\tau$ and set
\[
  \Ga=D_\tau\Ga_{\theta,\eta}\,.
\]
Clearly, the hyperbolic rotation $D_\tau$ preserves Diophantine exponents of a lattice, i.e.
\begin{equation}\label{eq:hyperbolic_rotation_preserves_omegas}
  \omega(\Ga)=
  \omega(\Ga_{\theta,\eta}).
\end{equation}
As for the hyperbolic minima of the lattice $\Ga_{\theta,\eta}$ w.r.t. the norm $|\,\cdot\,|$, they are mapped to the hyperbolic minima of the lattice $\Ga$ w.r.t. the norm $|\,\cdot\,|_\tau=|D_\tau(\,\cdot\,)|$. Thus, $\Ga_{\theta,\eta}$ satisfies condition \eqref{eq:beta_growth_of_successive_hyperbolic_minima} of Lemma \ref{l:from_bar_omega_to_bar_omega_L} if and only if so does $\Ga$.

Next, we construct a subspace $\cL$ and a lattice $\Ga_\cL$ as in Section \ref{sec:rank_2_sublattice_construction}. We choose $\bE=(\vec e_1,\ldots,\vec e_n)$ according to the statement of Lemma \ref{l:main_metric_argument} and set
\[
  \La=\La_\bE\,.
\]
Then by Lemma \ref{l:2dim_provides_all}
\begin{equation}\label{eq:2dim_provides_all}
  \omega(\La)=\max(0,\omega_\bL(\Ga)),\qquad
  \bar\omega(\La)=\max(0,\bar\omega_\bL(\Ga)),
\end{equation}
and we finish the proof by applying Lemmas \ref{l:from_omega_to_omega_L}, \ref{l:from_bar_omega_to_bar_omega_L}.

\paragraph{Choosing $\theta$ and $\eta$ to prove Theorem \ref{t:spectrum_of_the_regular_arbitrary_d}.}

Given $\beta\geq0$, we set $a_0=1$ and $a_{k+1}=[q_k^\beta]+1$, for each integer $k\geq0$, where $q_k$ is the denominator of the convergent $p_k/q_k=[a_0;a_1,a_2,\ldots,a_k]$. Consider the number $\theta=[a_0;a_1,a_2,\ldots]$ and the lattice $\Ga_{\theta,\theta}$\,. Then, as it was shown in \cite{german_sbornik_2026} (see also \cite{german_izv_2026}),
\[
  \omega(\Ga_{\theta,\theta})=\delta,
  \qquad\text{ where }\qquad
  \delta=\frac{\beta-1}2\,.
\]
By \eqref{eq:hyperbolic_rotation_preserves_omegas} we have $\omega(\Ga)=\omega(\Ga_{\theta,\theta})=\delta$. Thus, for $\delta\geq n/2$, i.e. for $\beta\geq n+1$, according to \eqref{eq:2dim_provides_all} and the statement of Lemma \ref{l:from_omega_to_omega_L}, we get
\[
  \omega(\La)=\omega_\bL(\Ga)=
  \dfrac{2\delta-n}{n+2}=
  \dfrac{\beta-n-1}{n+2}\,.
\]
Therefore, as $\beta$ ranges through $[n+1,+\infty]$, the exponent $\omega(\La)$ ranges through $[0,+\infty]$, which proves Theorem \ref{t:spectrum_of_the_regular_arbitrary_d}.

\paragraph{Choosing $\theta$ and $\eta$ to prove Theorem \ref{t:spectrum_of_the_weak_arbitrary_d}.}

Given $\beta>1$, we set $a_0=b_0=1$, $a_1=2$, $b_1=[2^\beta]+1$, and
\begin{equation}\label{eq:next_partial_qoutients_definition}
  a_k=\bigg[\frac{s_{k-1}^\beta-q_{k-2}}{q_{k-1}}\bigg]+1,
  \qquad
  b_k=\bigg[\frac{q_k^\beta-s_{k-2}}{s_{k-1}}\bigg]+1,
\end{equation}
for every integer $k\geq2$, where $q_k$ is the denominator of the convergent $p_k/q_k=[a_0;a_1,a_2,\ldots,a_k]$ and $s_k$ is the denominator of the convergent $r_k/s_k=[b_0;b_1,b_2,\ldots,b_k]$. Consider the numbers $\theta=[a_0;a_1,a_2,\ldots]$, $\eta=[b_0;b_1,b_2,\ldots]$ and the lattice $\Ga_{\theta,\eta}$\,. Then, as it was shown in \cite{german_sbornik_2026}, the sequence $(\vec x_k)_{k\in\N}$ of hyperbolic minima of $\Ga_{\theta,\eta}$ satisfies the relations
\[
  |\vec x_{k+1}|\asymp|\vec x_k|^\beta\,,
  \qquad
  \Pi(\vec x_k)\asymp|\vec x_k|^{(1-\beta^2)/2}.
\]
Thus, for $\beta\geq\sqrt{n+1}$, the lattice $\Ga_{\theta,\eta}$ satisfies condition \eqref{eq:beta_growth_of_successive_hyperbolic_minima} of Lemma \ref{l:from_bar_omega_to_bar_omega_L}. Then, as we mentioned before, $\Ga$ also satisfies condition \eqref{eq:beta_growth_of_successive_hyperbolic_minima} of Lemma \ref{l:from_bar_omega_to_bar_omega_L}. Hence, for $\beta\geq\sqrt{n+1}$, according to \eqref{eq:2dim_provides_all} and the statement of Lemma \ref{l:from_bar_omega_to_bar_omega_L}, we get
\[
  \bar\omega(\La)=\bar\omega_\bL(\Ga)=
  \dfrac{\beta-(n+1)\beta^{-1}}{n+2}\,.
\]
Therefore, as $\beta$ ranges through $\big[\sqrt{n+1},+\infty\big]$, the exponent $\bar\omega(\La)$ ranges through $[0,+\infty]$, which proves Theorem \ref{t:spectrum_of_the_weak_arbitrary_d}.

\section*{Acknowledgements}
The author is a winner of the ``Leader'' contest con\-duct\-ed by the Foundation for the Advancement of Theoretical Physics and Mathematics ``BASIS'' and would like to thank its sponsors and jury.

\end{document}